\documentclass{article}
\usepackage[utf8]{inputenc}
\usepackage[english]{babel}

\usepackage{amsmath,amsthm,amsfonts,amssymb,amscd}
\usepackage{lastpage}
\usepackage{enumerate}
\usepackage{fancyhdr}
\usepackage{mathrsfs}
\usepackage{polynom}
\usepackage{graphicx}
\usepackage{tikz}
\usepackage{hyperref}
\usepackage{csquotes}
\usepackage{verbatim}
\usepackage{accents}
\usepackage{pinlabel}
 
\usepackage{biblatex} 
\usepackage{geometry}
\newcommand{\N}{\mathbb{N}}
\newcommand{\Q}{\mathbb{Q}}

\newcommand{\arc}{\mathcal{A}(S)}
\newcommand{\arcp}{\mathcal{A}(S')}

\theoremstyle{plain}
\newtheorem{theorem}{Theorem}[section]
\newtheorem{lemma}[theorem]{Lemma}
\newtheorem{prop}[theorem]{Proposition}
\newtheorem{corollary}[theorem]{Corollary}

\theoremstyle{remark}
\newtheorem*{remark}{Remark}
\newtheorem*{acknowledgements}{Acknowledgements}

\begin{document}

\title{Finite Rigid Sets in Arc Complexes}
\author{Emily Shinkle} 
\date{}

\maketitle             

\begin{abstract} 
For any compact, connected, orientable, finite-type surface with marked points other than the sphere with three marked points, we construct a finite rigid set of its arc complex: a finite simplicial subcomplex of its arc complex such that any locally injective map of this set into the arc complex of another surface with arc complex of the same or lower dimension is induced by a homeomorphism of the surfaces, unique up to isotopy in most cases. It follows that if the arc complexes of two surfaces are isomorphic, the surfaces are homeomorphic. We also give an exhaustion of the arc complex by finite rigid sets. This extends the results of Irmak--McCarthy \cite{irmak2010injective}. 
\end{abstract}

\section{Introduction}

Let $S$, $S'$ be compact, connected, orientable, finite-type surfaces (with empty boundary) with marked points. The \emph{arc complex} $\arc$ of $S$, first defined by Harer in \cite{harer1986virtual}, is the simplicial complex whose vertices correspond to the isotopy classes of arcs on $S$ and whose $k$-simplices ($k>0$) correspond to collections of $k+1$ distinct isotopy classes of arcs which are pairwise disjoint. Homeomorphisms of $S$ induce simplicial automorphisms of $\arc$. Conversely, in \cite{irmak2010injective}, Irmak--McCarthy prove that every simplicial automorphism (in fact, every injective simplicial self-map) of $\arc$ arises from a homeomorphism of $S$, unique up to isotopy in most cases; see also \cite{disarlo2015combinatorial}. In these cases $\text{Aut}(\arc)$ is isomorphic to $\text{Mod}^\pm(S)$, the extended mapping class group of $S$.

Theorems of this type were first proved by Ivanov, who showed in \cite{ivanov1997automorphisms} that any automorphism of the curve complex of a surface with genus $g \geq 2$ is induced by a homeomorphism of the surface. Korkmaz \cite{korkmaz1999automorphisms} and Luo \cite{luo1999automorphisms} cover the cases where $g<2$. Further, Shackleton \cite{shackleton2007combinatorial} and Hernández-Hernández \cite{hernandez2018edge} give conditions such that maps between a priori different curve complexes are induced by homeomorphisms. Aramayona--Leininger construct finite rigid sets in the curve complex in \cite{aramayona2013finite} and provide an exhaustion of the curve complex by finite rigid sets in \cite{aramayona2016exhausting}. In this paper, we adapt the arguments of Irmak--McCarthy to prove analogues of Aramayona--Leininger's results for the arc complex. 

Recall that a simplicial map is \emph{locally injective} if the restriction to the star of every vertex is injective.

\begin{theorem}\label{main theorem one general}
Let $S=S_{g,n}$ be a compact, connected, orientable surface of genus $g$ with $n\geq 1$ marked points. If $S\neq S_{0,3}$, then there exists a finite simplicial subcomplex $\mathcal{X}$ of $\arc$ such that for any compact, connected, orientable, finite-type surface $S'$ with marked points and with $\text{dim}(\arc)\geq\text{dim}(\arcp)$ and for any locally injective simplicial map 
\[\lambda:\mathcal{X} \rightarrow \arcp,\]
there is a homeomorphism $H:S\rightarrow S'$ which induces $\lambda$. Moreover, if $S$ $\neq$ $S_{0,1}$, $S_{0,2}$, or $S_{1,1}$, $H$ is unique up to isotopy, and in these exceptional cases,
$H$ is unique up to $Z(\text{Mod}^\pm(S))$.
\end{theorem}

We refer to any simplicial subcomplex of $\arc$ with this property as \emph{rigid}. Theorem \ref{main theorem one general} implies the following generalization of Irmak--McCarthy's result (cf \cite[Theorem 1.1]{irmak2010injective}).

\begin{corollary}\label{generalization of im}
Let $S$ and $S'$ be compact, connected, orientable, finite-type surfaces with marked points such that $\text{dim}(\arc)\geq\text{dim}(\arcp)$ and $S\neq S_{0,3}$. Then for any locally injective simplicial map $\phi:\arc \rightarrow \arcp$, there is a homeomorphism $H:S\rightarrow S'$ which induces $\phi$. Moreover, if $S$ $\neq$ $S_{0,1}$, $S_{0,2}$, or $S_{1,1}$, $H$ is unique up to isotopy, and in these exceptional cases,
$H$ is unique up to $Z(\text{Mod}^\pm(S))$.
\end{corollary}

We provide a counterexample to demonstrate why Theorem \ref{main theorem one general} and Corollary \ref{generalization of im} do not hold for $S=S_{0,3}$. However, we use a cardinality argument to show that the following corollary holds even if $S=S_{0,3}$. 

\begin{corollary}\label{corollary to main theorem one}
Let $S$ and $S'$ be compact, connected, orientable, finite-type surfaces with marked points. If $\arc$ and $\arcp$ are isomorphic, then $S$ and $S'$ are homeomorphic.
\end{corollary}

Corollary \ref{corollary to main theorem one} was also shown by Bell--Disarlo--Tang in \cite{bell2019cubical} using a result about the
flip graph. Our proof relies only upon results about the arc complex. 

Finally, we give an exhaustion of $\arc$ by finite rigid sets.

\begin{theorem} \label{main theorem two}
For any compact, connected, orientable, finite-type surface $S$ with marked points such that $S\neq S_{0,3}$ there exists a sequence $(\mathcal{X}_i)_{i\in \N}$ of finite rigid sets of $\arc$ such that $\mathcal{X}_0\subseteq \mathcal{X}_1 \subseteq ... \subseteq \arc$ and $\bigcup_{i\in \N} \mathcal{X}_i = \arc$. 
\end{theorem}

\begin{remark}
Irmak--McCarthy prove their results for surfaces with $n$ boundary components rather than $n$ marked points; however, they consider arcs up to isotopy not necessarily fixing the endpoints of the arcs. This implies that Dehn twists around boundary components act trivially on $\arc$. Hence we can think of these boundary components as marked points instead. In \cite{disarlo2015combinatorial}, Disarlo considers surfaces with nonempty boundary, with at least one marked point on each boundary component and with a finite number of punctures in the interior. In contrast to Irmak--McCarthy, she considers arcs up to isotopy fixing the endpoints; hence, in this setting Dehn twists around boundary components act nontrivially on $\arc$. Disarlo proves that in this context, isomorphisms between arc complexes are induced by homeomorphisms of the associated surfaces. 
\end{remark}

Results of this type are often referred to as \emph{Ivanov-style rigidity} or simply \emph{rigidity}. Since Ivanov's seminal work, rigidity results have been proved for many other complexes, including the pants complex (Margalit \cite{margalit2004automorphisms}, Aramayona \cite{aramayona2010simplicial}), the arc and curve complex (Korkmaz--Papadopoulos \cite{korkmaz2010arc}), the ideal triangluation graph or flip graph (Korkmaz--Papadopoulos \cite{korkmaz2012ideal}, Aramayona--Koberda--Parlier \cite{aramayona2015injective}), the Schumtz graph of nonseparating curves (Schaller \cite{schaller2000mapping}), the complex of nonsepating curves (Irmak \cite{irmak2004complexes}), the Hatcher-Thurston complex (Irmak--Korkmaz \cite{irmak2007automorphisms}), and the polygonalisation complex (Bell--Disarlo--Tang \cite{bell2019cubical}), among others. See \cite{mccarthy2012simplicial} by McCarthy and Papadopoulus for a survey. Rigidity results exist for complexes of non-orientable surfaces as well, such as the curve complex (Atalan--Korkmaz \cite{atalan2014automorphisms}, Irmak \cite{irmak2014simplicial}, Irmak \cite{irmak2012superinjective}), the arc complex (Irmak \cite{irmak2008injective}), and the two-sided curve complex (Irmak--Paris \cite{irmak2017superinjective}). In \cite{ivanov2006fifteen}, Ivanov conjectured that every object naturally associated to a surface with sufficiently rich structure has the extended mapping class group as its group of automorphisms. Brendle--Margalit \cite{brendle2019normal} prove general rigidity results about subcomplexes of the complex of domains (introduced by McCarthy--Papadopoulos in \cite{mccarthy2012simplicial}) for closed surfaces and McLeay \cite{mcleay2018geometric} proves general rigidity results for subcomplexes of the complex of domains of punctured surfaces. These imply rigidity for a large class of complexes towards Ivanov's conjecture. 

Finite rigidity results are less plentiful but some exist. Maungchang \cite{maungchang2018finite} proves finite rigidity for the pants graph of a punctured sphere and Hernández-Hernández--Leininger--Maungchang extend the result for any surface $S_{g,n}$ in \cite{maungchang2019finite}. Maungchang also gives an exhaustion of the pants graph of punctured spheres by finite rigid sets in \cite{maungchang2017exhausting}. For non-orientable surfaces, Ilbira--Korkmaz construct finite rigid sets of the curve complex in \cite{ilbira2018finite} and Irmak gives an exhaustion of the curve complex by finite rigid sets in \cite{irmak2019exhausting2}, strengthening her exhaustion by finite superrigid sets in \cite{irmak2019exhausting}. 

\begin{acknowledgements}
The author would like to thank Christopher Leininger for suggesting this project and for providing guidance and support throughout it. She would also like to thank Marissa Miller and Dan Margalit for their helpful comments on a draft of the paper. Finally, she would like to thank the referee for their thoughtful feedback and suggestions. 
\end{acknowledgements}

\section{The Arc Complex}

Let $S=S_{g,n}$ denote a compact, connected, orientable surface (with empty boundary) of genus $g$ with $n\geq 1$ marked points. Let $\mathscr{P}_S$ be the set of marked points of $S$. Homeomorphisms between surfaces will be assumed to map marked points to marked points and isotopies between these homeomorphisms will be relative to the marked points. The \emph{extended mapping class group} $\text{Mod}^\pm(S)$ of $S$ is the group of isotopy classes of homeomorphisms of $S$.

An \emph{arc} on $S$ is a map $\gamma:[0,1]\rightarrow S$ such that $\gamma(0), \gamma(1)\in \mathscr{P}_S$, $\gamma((0,1))\cap \mathscr{P}_S=\emptyset$, and $\gamma|_{(0,1)}$ is injective. We will identify an arc $\gamma$ with its image $\gamma([0,1])$ on $S$ and we call $\gamma((0,1))$ the \emph{interior} of the arc. We call two arcs isotopic only if there exists an isotopy between them such that each of the transition maps is also an arc. In particular, isotopies of arcs will be relative to the endpoints and are not permitted to pass through marked points. We will assume that all arcs are \emph{essential}, ie cannot be isotoped into a regular neighborhood of a marked point.

If $a$ and $b$ are isotopy classes of arcs on $S$, then the \emph{geometric intersection number} $i(a,b)$, or \emph{intersection number}, is the minimum number of intersection points of the interiors of representatives of $a$ and $b$. The \emph{arc complex} $\arc$ is the simplicial complex whose vertices correspond to the isotopy classes of arcs on $S$ and whose $k$-simplices ($k>0$) correspond to collections of $k+1$ distinct isotopy classes of arcs which have pairwise disjoint representatives. Unless necessary, we will not distinguish between an isotopy class of arcs, a representative of the class, and the corresponding vertex of the arc complex. We say two arcs are \emph{disjoint} if $i(a,b)=0$. By \emph{distinct arcs}, we mean distinct isotopy classes of arcs. 

We state explicitly the results of Irmak--McCarthy as we will reference them later. Recall that $Z(G)$ refers to the center of the group $G$.

\begin{theorem} \label{irmak-mccarthy theorem 1} (\cite{irmak2010injective}, Theorem 1.1). Let $S$ be a compact, connected, orientable, finite-type surface with marked points. If $\varphi: \arc \rightarrow \arc$ is an injective simplicial map then $\varphi$ is induced by a homeomorphism $H:S\rightarrow S$. 
\end{theorem}

\begin{theorem} \label{irmak-mccarthy theorem 2} (\cite{irmak2010injective}, Theorem 1.2).
Let $S$ be a compact, connected, orientable, finite-type surface with marked points. If $S$ is not $S_{0,1}$, $S_{0,2}$, $S_{0,3}$, or $S_{1,1}$, then $\text{Aut}(\arc)$ is naturally isomorphic to $\text{Mod}^\pm(S)$. For each of the special cases, $\text{Aut}(\arc)$ is naturally isomoprhic to $\text{Mod}^\pm(S)/Z(\text{Mod}^\pm(S))$.
\end{theorem}

A key object in the proofs of these results and our results is triangulations of surfaces. We dedicate the remainder of this section to describing necessary background information regarding triangulations. If $S$ is neither $S_{0,1}$ nor $S_{0,2}$, it has at least three distinct, disjoint arcs, and we call a maximal collection of distinct pairwise disjoint arcs on $S$ a \emph{triangulation} of $S$. There is a natural $\Delta$-complex associated to a triangulation with the arcs as its 1-skeleton, hence the name ``triangulation'' (see eg Hatcher \cite[Section 2.1]{hatcher2001algebraic} for more on $\Delta$-complexes). We will refer to the 2-cells as \emph{triangles} of $T$. We say that an arc $a$ is a \emph{side} of a triangle $\Delta$ in $T$ if $a$ is contained in $\partial\Delta$. We say a triangle is \emph{embedded} if its sides are distinct arcs on $S$ and that it is \emph{non-embedded} otherwise. Note that an embedded triangle is not required to have distinct vertices. All non-embedded triangles are of the form pictured in Figure \ref{example of non-embedded triangle}. Call the side of a non-embedded triangle which joins two distinct punctures the \emph{inner arc} (eg arc $a$ in Figure \ref{example of non-embedded triangle}). Call the other arc the \emph{outer arc} (eg arc $b$ in Figure \ref{example of non-embedded triangle}). 

\begin{figure}[h]
    \labellist
        \pinlabel {$a$} at 355 325
        \pinlabel {$b$} at 535 145
    \endlabellist
    \centering
    \includegraphics[width=1in]{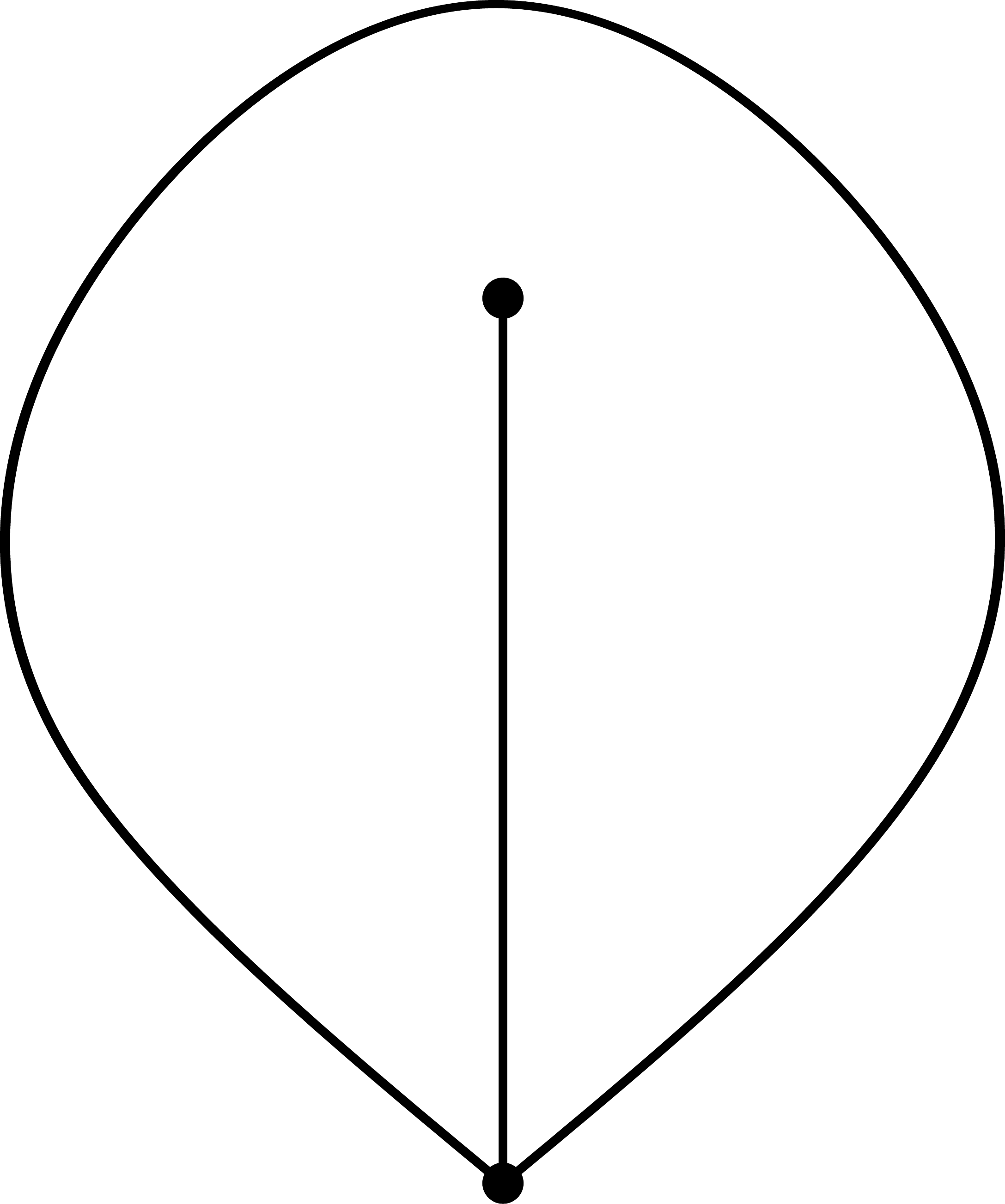}
    \caption{A non-embedded triangle}
    \label{example of non-embedded triangle}
\end{figure}

We can use the Euler characteristic to see that the number of arcs in a triangulation of $S$ is $6g+3n-6$, hence $\text{dim}(\arc)=6g+3n-7$. If $\text{dim}(\arc)=\text{dim}(\arcp)$, a locally injective simplicial map from a subcomplex $\mathcal{Y}$ of $\arc$ into $\arcp$ sends any triangulation $T$ contained in $\mathcal{Y}$ to a triangulation $T'$ of $S'$. The number of triangles in a triangulation is $4g+2n-4$. Hence, if $\text{dim}(\arc)=\text{dim}(\arcp)$, then a triangulation of $S$ has the same number of triangles as a triangulation of $S'$. 

We say that two triangulations are obtained from each other by a \emph{flip} if they differ by exactly one arc. In this case, the distinct arcs have intersection number one. Conversely, if two distinct arcs $a$ and $b$ have intersection number one, there exist triangulations $T_a$ and $T_b$ containing $a$ and $b$, respectively, such that $T_a\backslash \{a\}=T_b\backslash \{b\}$.

We will need the following result of Mosher regarding triangulations later.

\begin{prop} \label{prop: connectedness} (\cite{mosher1988tiling}, ``Connectivity Theorem for Elementary Moves''). Let $S$ be a compact, connected, orientable, finite-type surface with marked points. Any two triangulations of $S$ differ by a finite number of flips. 
\end{prop}

\section{Exceptional Cases}\label{section: exceptional cases}

In this section, we dispense with the cases where $\arc$ is empty or has dimension $\leq 2$, ie if $S$ is $S_{0,1}$, $S_{0,2}$, $S_{0,3}$, or $S_{1,1}$. On $S_{0,1}$, there are no essential arcs, hence $\mathcal{A}(S_{0,1})=\emptyset$. On $S_{0,2}$, there is only one essential arc, hence $\mathcal{A}(S_{0,2})$ is a single point and has dimension 0. All other surfaces have arc complex of dimension $>0$, hence for $S_{0,1}$ and $S_{0,2}$, Theorem \ref{main theorem one general} follows from setting $\mathcal{X}=\arc$ and applying the results of Irmak--McCarthy (Theorems \ref{irmak-mccarthy theorem 1} and \ref{irmak-mccarthy theorem 2}).

Now suppose $\text{dim}(\arc)=2$. Then $S$ is $S_{0,3}$ or $S_{1,1}$. It is well-known that $\mathcal{A}(S_{0,3})$ is isomorphic to a regular tessellation of a triangle by four triangles (see eg. \cite{irmak2010injective}, \cite{mccarthy2012simplicial}) and that $\mathcal{A}(S_{1,1})$ is isomorphic to the flag complex of the \emph{Farey graph}, a decomposition by ideal triangles of $\mathbb{H}^2 \cup \Q \mathbb{P}^1$ (see eg. \cite{irmak2010injective}). Figure \ref{complexity three counter} shows a simplicial embedding of $\mathcal{A}(S_{0,3})$ into $\mathcal{A}(S_{1,1})$ but $S_{0,3}$ and $S_{1,1}$ are not homeomorphic. Hence, Theorem \ref{main theorem one general} and Corollary \ref{generalization of im} cannot be extended for $S_{0,3}$. However, we can find a finite rigid set in $\mathcal{A}(S_{1,1})$ and we prove later that Corollary \ref{generalization of im} holds for $S_{1,1}$ as well. 

\begin{figure}[h]
    \centering
    \includegraphics[scale=.11]{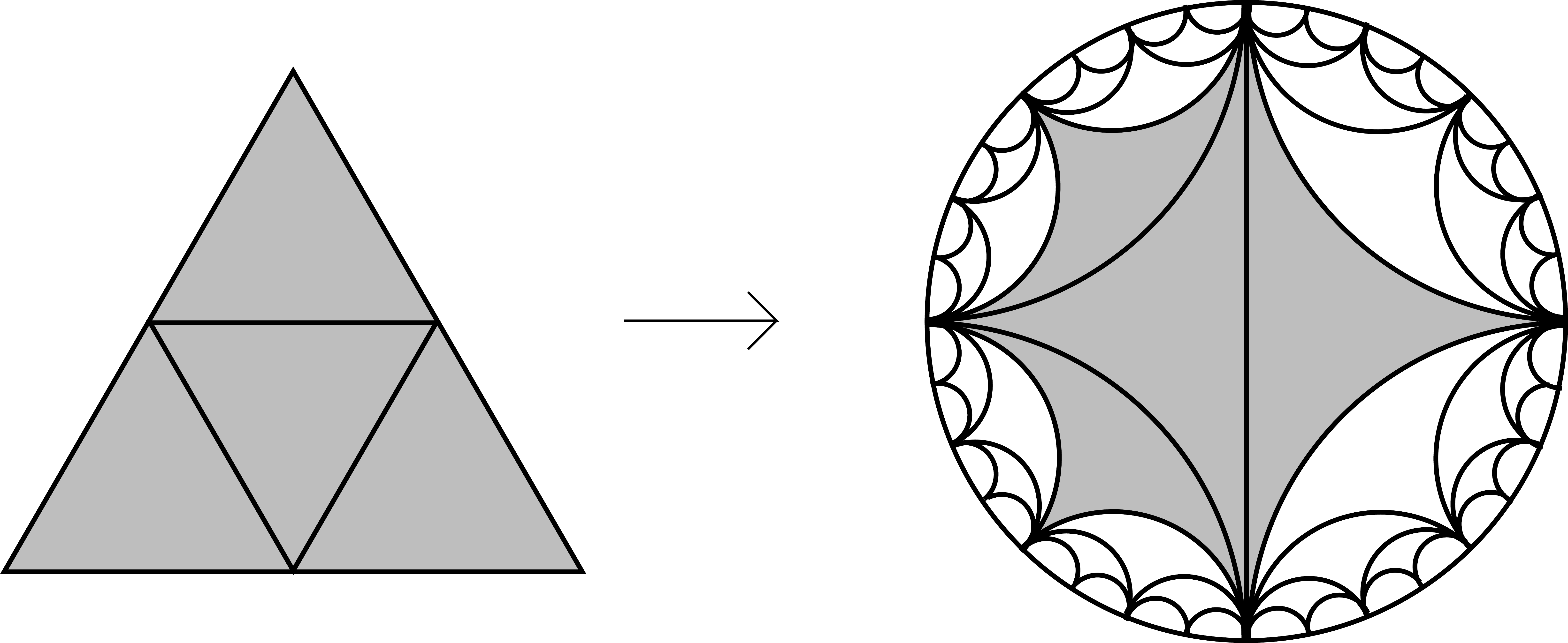}
    \caption{A simplicial embedding of $\mathcal{A}(S_{0,3})$ into $\mathcal{A}(S_{1,1})$}
    \label{complexity three counter}
\end{figure}

\begin{figure}[h]
    \labellist
        \pinlabel {$a$} at 420 0
    \endlabellist
    \centering
    \includegraphics[scale=.11]{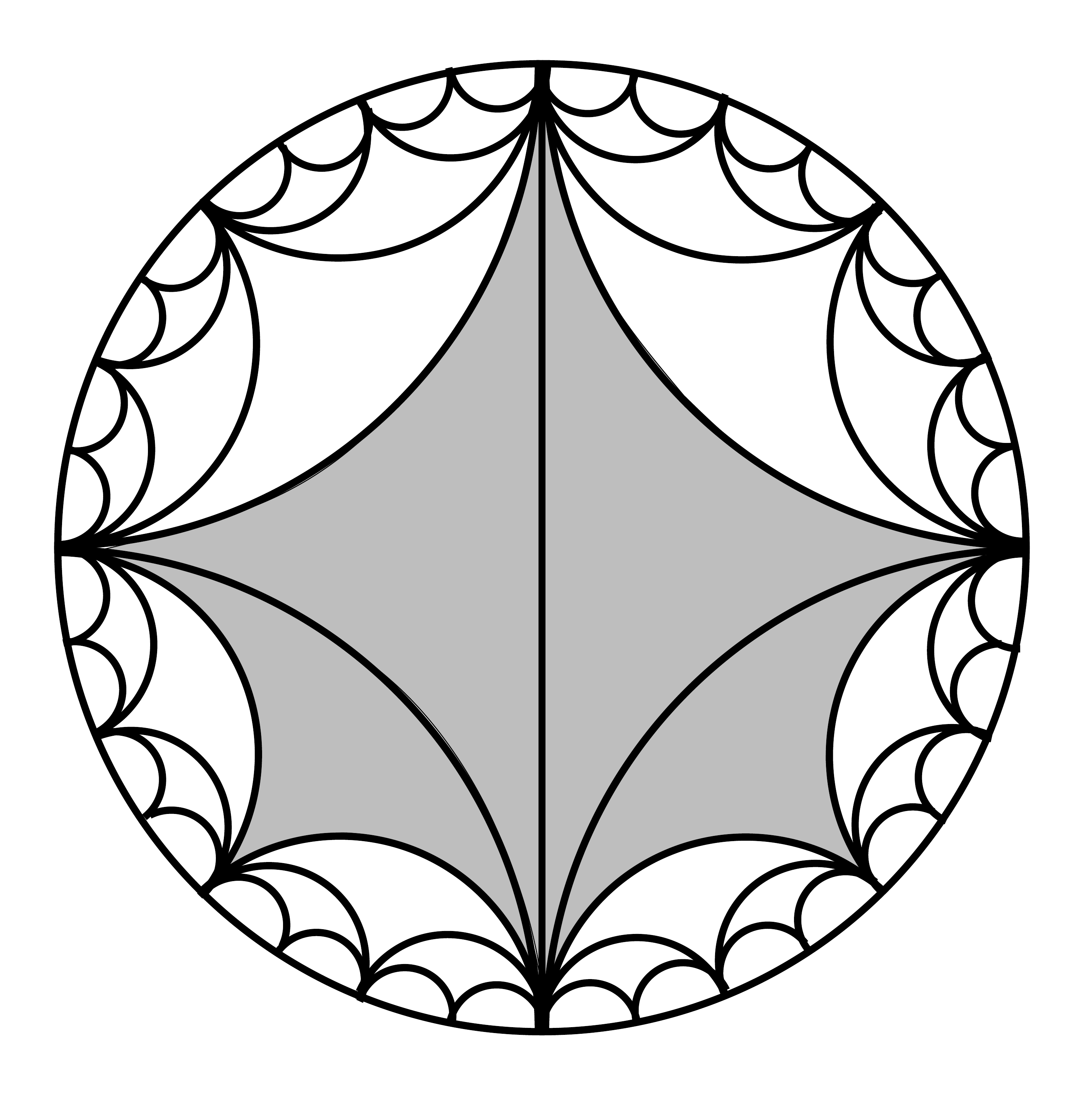}
    \caption{A rigid set in $\mathcal{A}(S_{1,1})$}
    \label{rigid set in farey graph}
\end{figure}

\begin{proof}[Proof of Theorem \ref{main theorem one general} for $S=S_{1,1}$.]

Let $\mathcal{X}$ be the simplicial subcomplex of $\arc$ indicated in grey in Figure \ref{rigid set in farey graph}. Let $S'$ be such that $\text{dim}(\arc)\geq\text{dim}(\arcp)$ and let $\lambda:\mathcal{X}\rightarrow \arcp$ be a locally injective simplicial map. Observe that $\mathcal{X}$ is contained in the closure of the star of the vertex marked $a$, hence $\lambda$ is injective. Since $\lambda$ inectively maps a simplex of dimension 2 into $\arcp$, $\arcp$ must not have dimension less than 2. Further, since there are five vertices in $\mathcal{X}$ connected to $a$ by an edge in $\mathcal{X}$ and since every vertex in $\mathcal{A}(S_{0,3})$ has degree at most four, $S'\neq S_{0,3}$, hence $S'=S_{1,1}$. Now, $\arc$ is connected and each edge borders exactly two triangles, so an induction argument shows that $\lambda$ can be extended uniquely to to an automorphism of $\arc$. Then we again apply the results of Irmak--McCarthy.

\end{proof}

Now we can give an exhaustion of $\arc$ by finite rigid sets in the cases that $S$ is $S_{0,1}$, $S_{0,2}$, and $S_{1,1}$.

\begin{proof}[Proof of Theorem \ref{main theorem two} for $S_{0,1}$, $S_{0,2}$, and $S_{1,1}$.]
Suppose $S$ is $S_{0,1}$ or $S_{0,2}$. Then as discussed above, $\arc$ is finite, so we can take $\mathcal{X}_i=\arc$ for all $i\in \N$. Now suppose $S$ is $S_{1,1}$. Let $\mathcal{X}_0$ be the finite rigid set of $\arc$ from the proof above. For $i\in \N$, let $\mathcal{X}_{i+1}$ be the simplicial subcomplex of $\arc$ containing $\mathcal{X}_{i}$ and any triangles which share a side with a triangle in $\mathcal{X}_i$. Then $(\mathcal{X}_i)_{i\in \N}$ is an exhaustion of $\arc$ by finite rigid sets by the same argument as above.
\end{proof}

\section{The General Case}

In this section, $S$ and $S'$ will be surfaces with $\text{dim}(\arc) \geq \text{dim}(\arcp)$ and $\text{dim}(\arc) > 2$. Throughout the section, we frequently assume the existence of a local injection from some subset of $\arc$ into $\arcp$. Note that if if this subset contains a maximal simplex of $\arc$, ie a triangulation of $S$, that the dimension of $\arcp$ cannot be less than that of $\arc$, hence $\text{dim}(\arc)=\text{dim}(\arcp)$. We will utilize this fact without making any further note of it.

Suppose $V=\{\sigma_1, \sigma_2, \ldots, \sigma_n\}$ is a collection of simplices of a simplicial complex $C$. The \emph{span of $V$}, $\text{Span}_C(V)$, refers to the set of all simplices $\tau$ of $C$ such that each vertex of $\tau$ is the vertex of a simplex in $V$. Observe that $V$ finite implies that $\text{Span}_C(V)$ is finite as well. 

Roughly speaking, the proof of Theorem \ref{main theorem one general} will proceed as follows: We will include in $\mathcal{X}$ (the span of) a triangulation $T$ of $S$ which maps to a triangulation $T'$ of $ S'$ under any locally injective map. We first show that, by adding finitely many arcs to $\mathcal{X}$, we can guarantee that triangles of $T$ map to triangles of the same type (embedded or non-embedded) in $T'$. Then we show that, by adding finitely many more arcs to $\mathcal{X}$, we can guarantee that the orientations of adjacent triangles in $T'$ match so that the map $T\rightarrow T'$ can be extended to a homeomorphism $H:S\rightarrow S'$. We use Proposition \ref{prop: connectedness} to show that by including finitely many more arcs in $\mathcal{X}$ we can guarantee that any locally injective simplicial map $\lambda:\mathcal{X}\rightarrow \arcp$ agrees with the induced map $H_*$ on all of $\mathcal{X}$. Finally, we show uniqueness by proving that any other such homeomorphism $H'$ has induced map equal to $H_*$ and then applying the results of Irmak--McCarthy.

First, we need the following (cf \cite[Prop. 3.1]{irmak2010injective}).

\begin{lemma} \label{detectible}
If $(a,b)$ is a pair of arcs with $i(a,b)=1$, then there exists a finite simplicial subcomplex $\mathcal{B}$ of $\arc$ containing $a$ and $b$ with the following property: If $\mathcal{Y}$ is a simplicial subcomplex of $\arc$ which contains $\mathcal{B}$ and $\lambda: \mathcal{Y} \rightarrow \arcp$ a locally injective simplicial map, then $i(\lambda(a), \lambda(b)) = 1$. 
\end{lemma}

\begin{proof}
There exist triangulations $T_a$ and $T_b$ of $S$ which share all arcs except for $a\in T_a$ and $b\in T_b$. Let $T_0=T_a\cap T_b$. We define:
\[\mathcal{B} = \text{Span}_{\arc}(T_0 \cup \{a,b\}).\]

Now suppose $\mathcal{Y}$ is a simplicial subcomplex of $\arc$ which contains $\mathcal{B}$ and $\lambda: \mathcal{Y} \rightarrow \arcp$ a locally injective simplicial map. Note that $\lambda|_\mathcal{B}$ is injective since $\lambda$ is locally injective. Since the simplex spanned by $T_0$ has codimension one in $\arc$, $T_0'=\lambda(T_0)$ has codimension one in $\arcp$. Let $a'=\lambda(a)$ and $b'=\lambda(b)$. Then $\lambda(T_a)=\{a'\} \cup T_0'$ and $\lambda(T_b)=\{b'\} \cup T_0'$ are both triangulations of $S'$. It follows that $i(a',b')=1$. 
\end{proof}

We will also use the following result of Irmak--McCarthy:

\begin{prop}\label{nice triangulations} (\cite{irmak2010injective}, Proposition 3.2).
Let $\Delta$ be an embedded triangle on $S$ with sides $a$, $b$, and $c$. Then there exists a triangulation $T$ on $S$ containing $a$, $b$, and $c$ such that the unique triangles $\Delta_a, \Delta_b$, and $\Delta_c$ of $T$ on $S$ which are different from $\Delta$ and have, respectively, $a$ as a side, $b$ as a side, and $c$ as a side, are distinct triangles of $T$ on $S$.
\end{prop}

Applying Proposition \ref{nice triangulations}, we deduce the following (cf \cite[Prop. 3.3]{irmak2010injective}).

\begin{lemma}\label{embedded}
If $a$, $b$, and $c$ are the edges of an embedded triangle on $S$, then there exists a finite simplicial subcomplex $\mathcal{C}$ of $\arc$ containing $a$, $b$, and $c$ with the following property: If $\mathcal{Y}$ is a simplicial subcomplex of $\arc$ which contains $\mathcal{C}$ and $\lambda: \mathcal{Y} \rightarrow \arcp$ a locally injective simplicial map, then $\lambda(a)$, $\lambda(b)$, and $\lambda(c)$ are the edges of an embedded triangle on $S'$. 
\end{lemma}

\begin{proof}

By Proposition \ref{nice triangulations}, there exists a triangulation $T$ of $S$ containing $\{a,b,c\}$ so that the unique triangles $\Delta_a, \Delta_b$, and $\Delta_c$ of $T$ on $S$ which are different from $\Delta$ and have, respectively, $a$ as a side, $b$ as a side, and $c$ as a side, are distinct triangles of $T$ on $S$.

Let $P_1$ be the vertex of $\Delta_a$ opposite $a$, $P_2$ the vertex of $\Delta_b$ opposite $b$, and $P_3$ the vertex of $\Delta_c$ opposite $c$ (see Figure \ref{double triangle}). There exists an arc $d$ connecting $P_1$ to $P_2$ which intersects $a$ and $b$ once and is disjoint from all other arcs in $T$. Similarly, there exists an arc $e$ connecting $P_2$ to $P_3$ which intersects $b$ and $c$ once and is disjoint from all other arcs in $T$. And finally, there is an arc $f$ connecting $P_1$ to $P_3$ which intersects $a$ and $c$ once and is disjoint from all other arcs in $T$. Note that $d$, $e$, and $f$ are pairwise disjoint.

\begin{figure}[h]
    \centering
    \labellist
        \pinlabel {$P_1$} at 25 950
        \pinlabel {$P_2$} at 1025 950
        \pinlabel {$P_3$} at 550 20
        \pinlabel {$a$} at 455 675 
        \pinlabel {$b$} at 610 675
        \pinlabel {$c$} at 535 540 
        \pinlabel {$d$} at 410 820 
        \pinlabel {$e$} at 780 675
        \pinlabel {$f$} at 450 375 
    \endlabellist
    \includegraphics[scale=.11]{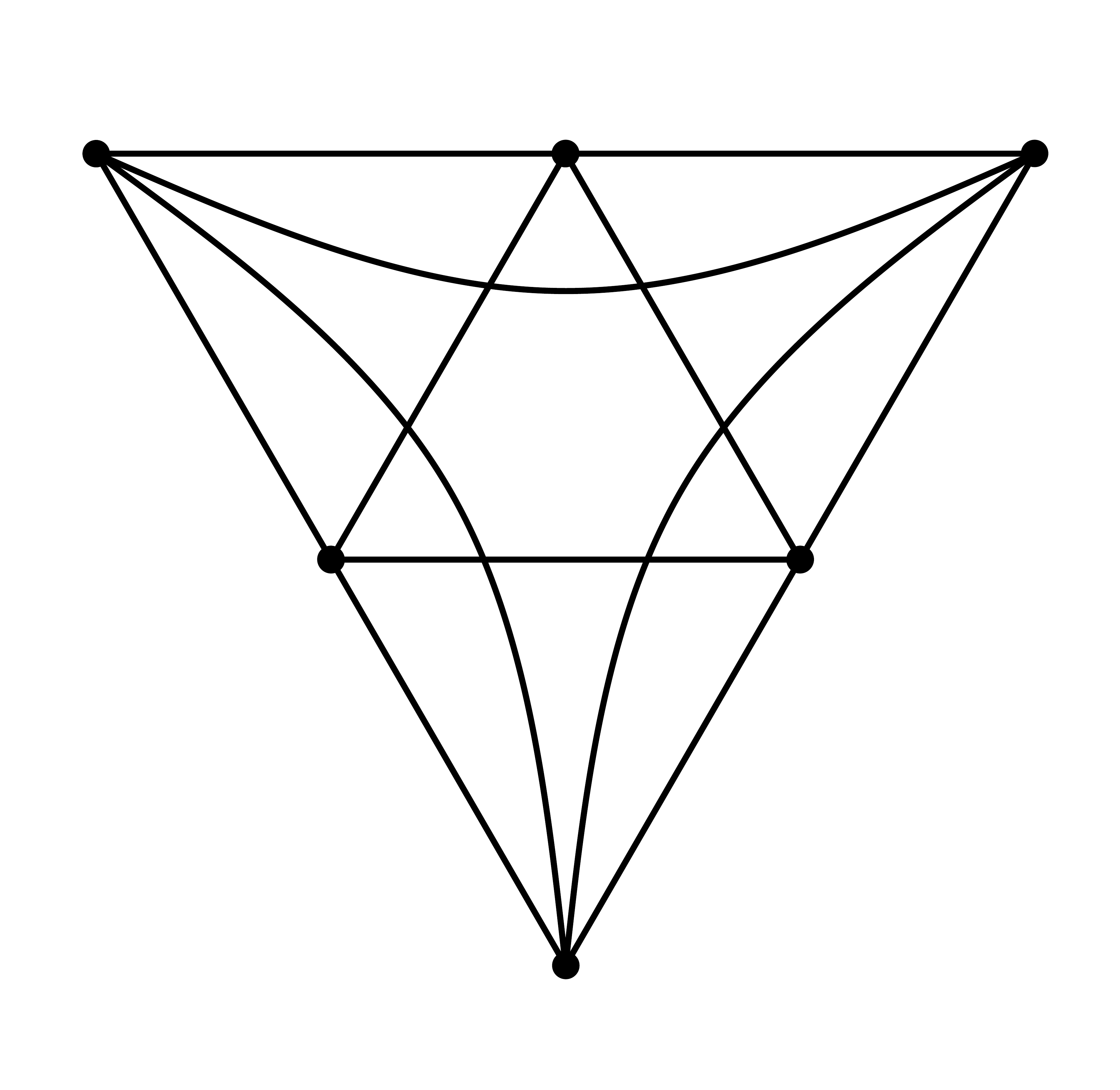}
    \caption{Arc configurations. Vertices and arcs around the perimeter may be identified.}
    \label{double triangle}
\end{figure}
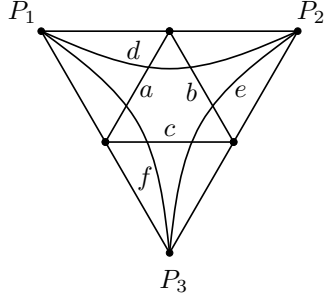

Now, consider the following six pairs of arcs: $\mathscr{R}_1 = (a,d),$ $\mathscr{R}_2 =(a,f),$ $\mathscr{R}_3 =(b,d),$ $\mathscr{R}_4 =(b,e),$ $\mathscr{R}_5 =(c,e),$ and $\mathscr{R}_6 =(c,f)$. Recall that each of these pairs has intersection number one. Let $\mathcal{B}_i$ be the finite simplicial complex of $\arc$ from Lemma \ref{detectible} corresponding to $\mathscr{R}_i$ for each $1\leq i \leq 6$.

We define:
\[\mathcal{C} = \text{Span}_{\arc}\left(T\cup \{d,e,f\} \cup \bigcup_{1\leq i \leq 6} \mathcal{B}_i\right).\]

Suppose $\mathcal{Y}$ is any simplicial subcomplex of $\arc$ which contains $\mathcal{C}$ and $\lambda: \mathcal{Y} \rightarrow \arcp$ is a locally injective simplicial map. Let $T' = \lambda(T)$; note that it is a triangulation of $S'$. Additionally, let $a'=\lambda(a)$, $b'=\lambda(b)$, $c'=\lambda(c)$, $d'=\lambda(d)$, $e'=\lambda(e)$, and $f'=\lambda(f)$. Lemma \ref{detectible} and the local injectivity of $\lambda$ guarantee each pair in $\{a', b', c', d', e', f'\}$ has the same intersection number as its preimage under $\lambda$. 

Since $d'$ intersects $a'$ and $b'$ once each and is disjoint from all other arcs in the triangulation $T'$, it must be the case that $a'$ and $b'$ border an embedded triangle in $T'$ --- call it $\Delta_1$. Analogously, since $e'$ intersects $b'$ and $c'$ once each and is disjoint from all other arcs in the triangulation $T'$, there is an embedded triangle $\Delta_2$ in $T'$ with sides $b'$ and $c'$ and since $f'$ intersects $a'$ and $c'$ once each and is disjoint from all other arcs in the triangulation $T'$, there is an embedded triangle $\Delta_3$ in $T'$ with sides $a'$ and $c'$. Suppose that the third side of $\Delta_1$ is $r'$, the third side of $\Delta_2$ is $s'$, and the third side of $\Delta_3$ is $t'$. If $r'=c'$, $s'=a'$, or $t'=b'$, then we are done. 

Suppose $r'\neq c'$, $s' \neq a'$, and $t' \neq b'$, hence $\Delta_1\neq \Delta_2$, $\Delta_2 \neq \Delta_3$ and $\Delta_3 \neq \Delta_1$. Up to homeomorphism (and ignoring admissible identifications among arcs and among vertices of the triangles), there are four configurations, and these depend on the relative orientations of $\Delta_1$ and $\Delta_2$, and the relative orientation of $\Delta_3$ with respect to $\Delta_1$ and $\Delta_2$. See Figure \ref{triangles for embedded case}. In each of these cases, $d'$ and $e'$ must intersect, which is a contradiction. 
\end{proof}

\begin{figure}[h]
    \labellist
        \pinlabel{(i)} at 325 20
        \pinlabel{(ii)} at 1300 20
        \pinlabel{(iii)} at 2275 20
        \pinlabel{(iv)} at 3250 20
        \pinlabel{$a'$} at 290 1640
        \pinlabel{$r'$} at 585 1640
        \pinlabel{$s'$} at 585 1220
        \pinlabel{$t'$} at 230 1120
        \pinlabel{$a'$} at 75 1270
        \pinlabel{$c'$} at 375 1275
        \pinlabel{$b'$} at 440 1480
        \pinlabel{$a'$} at 1255 1640
        \pinlabel{$r'$} at 1550 1640
        \pinlabel{$s'$} at 1550 1220
        \pinlabel{$a'$} at 1195 1120
        \pinlabel{$t'$} at 1040 1270
        \pinlabel{$c'$} at 1340 1275
        \pinlabel{$b'$} at 1405 1480
        \pinlabel{$a'$} at 2020 1640
        \pinlabel{$r'$} at 2315 1640
        \pinlabel{$c'$} at 2220 1285
        \pinlabel{$t'$} at 2360 1120
        \pinlabel{$s'$} at 2000 1220
        \pinlabel{$a'$} at 2505 1285
        \pinlabel{$b'$} at 2170 1480
        \pinlabel{$a'$} at 2985 1640
        \pinlabel{$r'$} at 3280 1640
        \pinlabel{$c'$} at 3185 1285
        \pinlabel{$a'$} at 3325 1120
        \pinlabel{$s'$} at 2960 1220
        \pinlabel{$t'$} at 3470 1285
        \pinlabel{$b'$} at 3135 1480
        \pinlabel{$a'$} at 220 630
        \pinlabel{$r'$} at 585 820
        \pinlabel{$s'$} at 585 200
        \pinlabel{$t'$} at 100 180
        \pinlabel{$e'$} at 230 735
        \pinlabel{$c'$} at 375 330
        \pinlabel{$b'$} at 490 650
        \pinlabel{$d'$} at 510 400
        \pinlabel{$c'$} at 1320 630
        \pinlabel{$s'$} at 1550 800
        \pinlabel{$t'$} at 1350 495
        \pinlabel{$r'$} at 1065 180
        \pinlabel{$e'$} at 1185 580
        \pinlabel{$a'$} at 1320 360
        \pinlabel{$b'$} at 1035 550
        \pinlabel{$d'$} at 1185 410
        \pinlabel{$a'$} at 2275 615
        \pinlabel{$t'$} at 2515 800
        \pinlabel{$r'$} at 2200 495
        \pinlabel{$s'$} at 2030 180
        \pinlabel{$c'$} at 2370 540
        \pinlabel{$b'$} at 2280 385
        \pinlabel{$d'$} at 2290 240
        \pinlabel{$e'$} at 2380 410
        \pinlabel{$b'$} at 3250 625
        \pinlabel{$r'$} at 3440 800
        \pinlabel{$s'$} at 3270 495
        \pinlabel{$t'$} at 2995 180
        \pinlabel{$e'$} at 3115 720
        \pinlabel{$c'$} at 3220 365
        \pinlabel{$d'$} at 3120 590
        \pinlabel{$a'$} at 3120 450
    \endlabellist
    \centering
    \includegraphics[scale=.11]{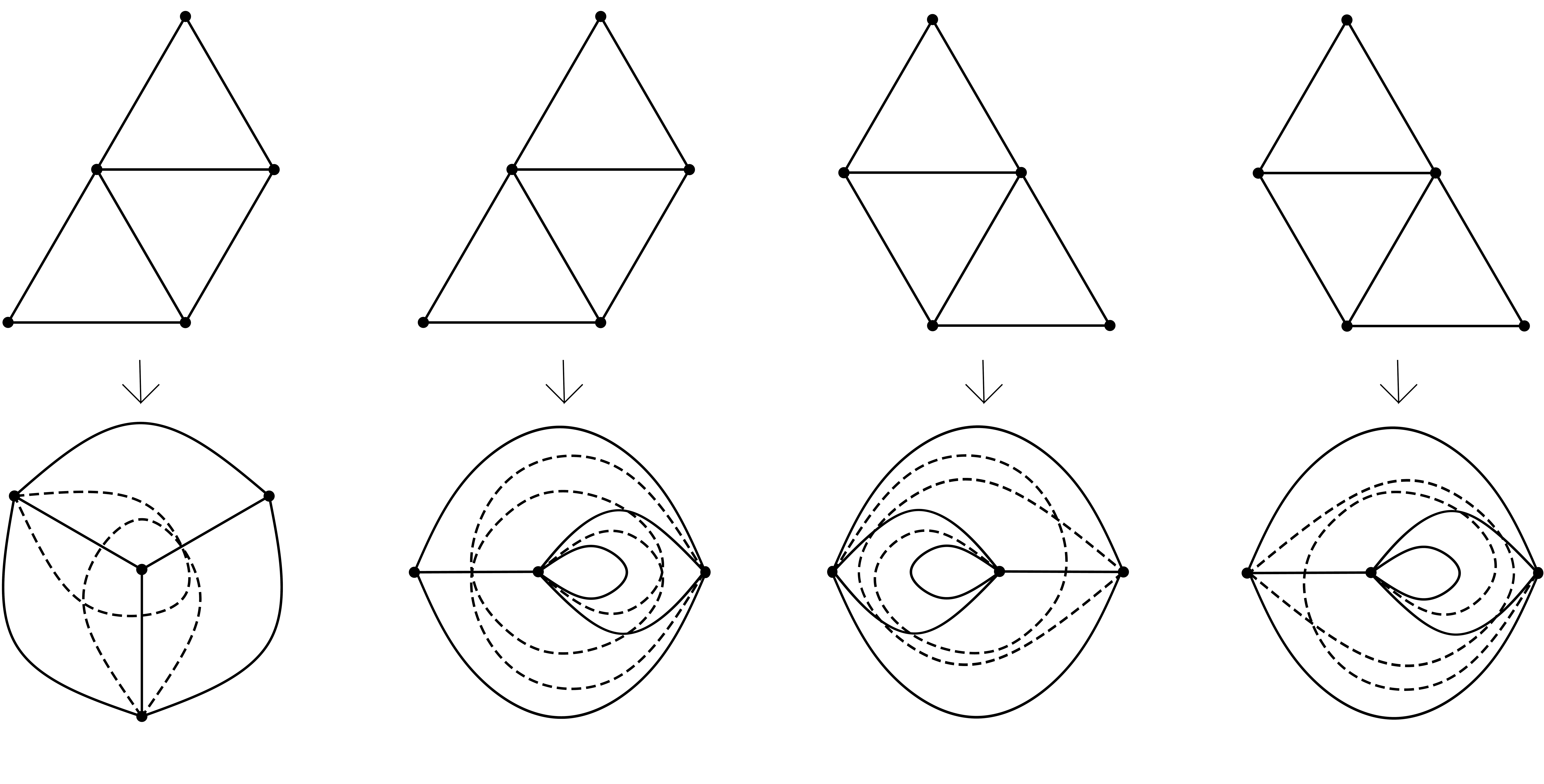}
    \caption{Four possible arrangements of triangles in the proof of Lemma \ref{embedded}. Some additional identifications of vertices and arcs may be made.}
    \label{triangles for embedded case}
\end{figure}
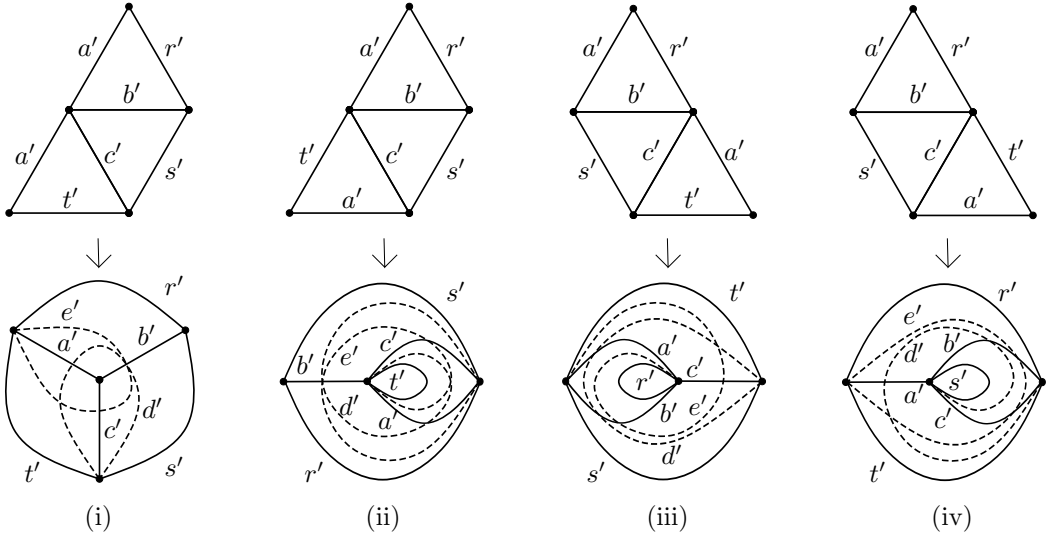

Now we can use Lemma \ref{embedded} to prove the corresponding result for non-embedded triangles (cf \cite[Prop. 3.4]{irmak2010injective}).

\begin{lemma}\label{non-embedded}
If $a$ and $b$ border a non-embedded triangle on $S$ with $a$ the inner arc, then there exists a finite simplicial subcomplex $\mathcal{D}$ of $\arc$ containing $a$ and $b$ with the following property: If $\mathcal{Y}$ is a simplicial subcomplex of $\arc$ which contains $\mathcal{D}$ and $\lambda: \mathcal{Y} \rightarrow \arcp$ a locally injective simplicial map, then $\lambda(a)$ and $\lambda(b)$ border a non-embedded triangle on $S'$ with $\lambda(a)$ as the inner arc. 
\end{lemma}

\begin{proof}

Since $S$ is not $S_{0,3}$, there exists an embedded triangle having $b$ as a side. Call the other sides of this triangle $c$ and $d$. Let $e$ be the arc pictured in Figure \ref{non-embedded image}. 

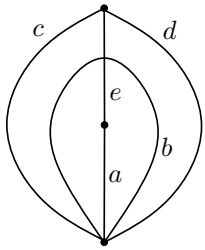
\begin{figure}[h]
    \labellist
        \pinlabel{$a$} at 465 250
        \pinlabel{$e$} at 465 530
        \pinlabel{$b$} at 640 350
        \pinlabel{$c$} at 200 750
        \pinlabel{$d$} at 650 750
    \endlabellist
    \centering
    \includegraphics[scale=.11]{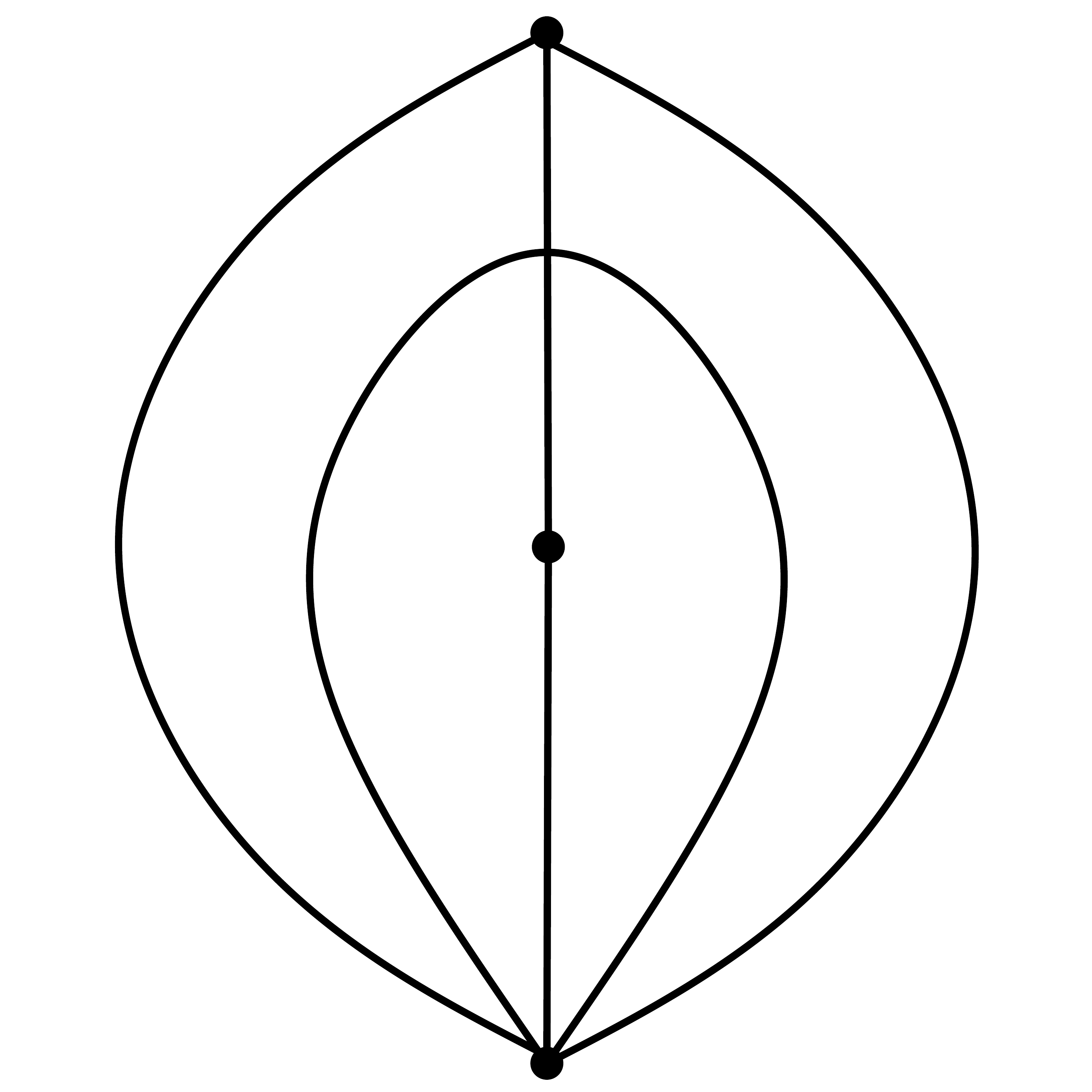}
    \caption{Arc configurations. Outer vertices may be identified. }
    \label{non-embedded image}
\end{figure}

We can see that arcs $a$, $c$, and $e$ border an embedded triangle on $S$, as do arcs $a$, $d$, and $e$. As noted previously, the arcs $b$, $c$, and $d$ also border an embedded triangle. Let $\mathcal{C}_1$, $\mathcal{C}_2$, and $\mathcal{C}_3$, respectively, be the finite simplicial subcomplexes of $\arc$ from Lemma \ref{embedded} corresponding to these triples of arcs. Also, observe that $i(b,e)=1$. Then let $\mathcal{B}$ be the corresponding finite simplicial subcomplex of $\arc$ from Lemma \ref{detectible}. We define:
\[\mathcal{D} = \text{Span}_{\arc}(\{a,b,c,d,e\}\cup \mathcal{B} \cup \mathcal{C}_1 \cup \mathcal{C}_2 \cup \mathcal{C}_3).\]

Let $\mathcal{Y}$ be any simplicial subcomplex of $\arc$ which contains $\mathcal{D}$ and $\lambda: \mathcal{Y} \rightarrow \arcp$ a locally injective simplicial map. Let $a'=\lambda(a)$, $b'=\lambda(b)$, $c'=\lambda(c)$, $d'=\lambda(d)$, and  $e'=\lambda(e)$. Lemma \ref{detectible} and the local injectivity of $\lambda$ guarantee each pair in $\{a', b', c', d', e'\}$ has the same intersection number as the preimage under $\lambda$. By Lemma \ref{embedded}, we know that $a'$, $c'$, and $e'$ border an embedded triangle on $S'$, as do $a'$, $d'$, and $e'$, and as do $b'$, $c'$, and $d'$. Since $b'$ is disjoint from $a'$, $c'$, and $d'$ and intersects $e'$ once, the only possible arrangement of arcs then guarantees that $a'$ and $b'$ border a non-embedded triangle on $S'$ with $a'$ as the inner arc.
\end{proof}

The two lemmas above give conditions such that a locally injective simplicial map takes a triangulation $T$ of $S$ to a triangulation $T'$ of $S'$ in such a way that triangles in $T$ are sent to triangles of the same type (embedded or non-embedded) in $T'$. The following lemma provides conditions under which two triangles in $T$ which share an edge, map to consistently oriented triangles in $T'$ (cf \cite[Prop. 3.5--3.7]{irmak2010injective}).

If $\Delta$ is a triangle, let $\accentset{\circ}{\Delta}$ denote $\Delta\backslash \mathscr{P}_S$, the triangle minus its vertex set. The edges of $\accentset{\circ}{\Delta}$ are the interiors of arcs; however, for simplicity, we will not use a separate notation for them.

\begin{figure}[h]
    \labellist
        \pinlabel{$a$} at 230 640
        \pinlabel{$b$} at 630 640
        \pinlabel{$d$} at 230 200
        \pinlabel{$e$} at 630 210
        \pinlabel{$f$} at 370 560
        \pinlabel{$c$} at 550 470
    \endlabellist
    \centering
    \includegraphics[scale=.11]{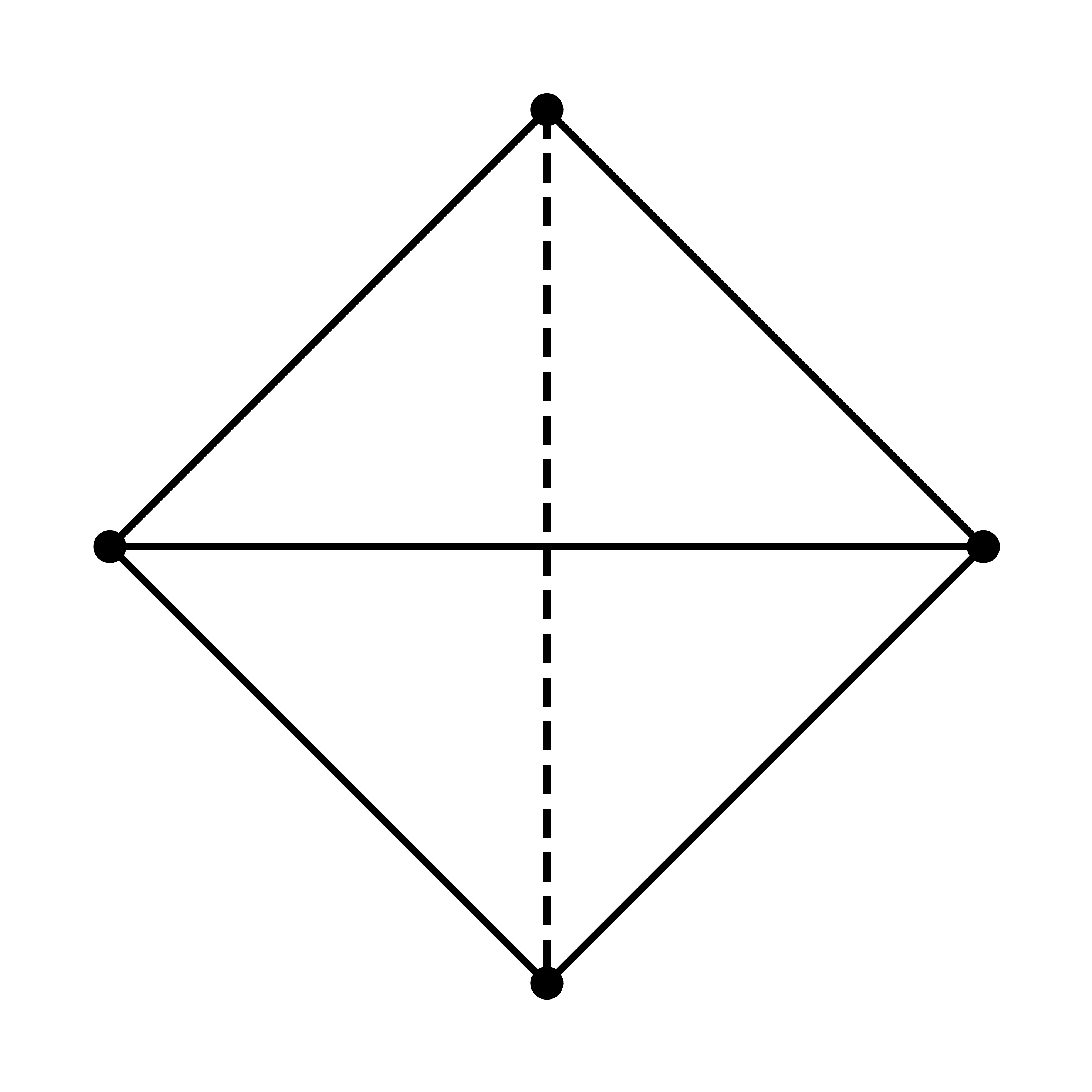}
    \caption{Two embedded triangles sharing a side. We allow for the possibility that $a=d$, $a=e$, $b=d$ or $b=e$. We also allow identifications among vertices.}
    \label{one side}
\end{figure}
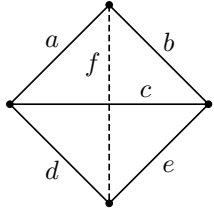

\begin{lemma}\label{one side prop}
Suppose $\Delta_1$ is an embedded triangle on $S$ with sides $a$, $b$, and $c$, and $\Delta_2$ an embedded triangle with sides $c$, $d$, and $e$, as shown in Figure \ref{one side}. Then there exists a finite simplicial subcomplex $\mathcal{E}$ of $\arc$ containing $a$, $b$, $c$, $d$, and $e$ with the following property: Let $\mathcal{Y}$ be any simplicial subcomplex of $\arc$ containing $\mathcal{E}$ and $\lambda:\mathcal{Y} \rightarrow \arcp$ be a locally injective simplicial map. Let $a'=\lambda(a)$, $b'=\lambda(b)$, $c'=\lambda(c)$, $d'=\lambda(d)$, and $e'=\lambda(e)$. Then there exist an embedded triangle $\Delta_1'$ on $S'$ with sides $a'$, $b'$, and $c'$, an embedded triangle $\Delta_2'$ on $S'$ with sides $c'$, $d'$, and $e'$, and the natural homeomorphisms $F_1:(\accentset{\circ}{\Delta}_1, a,b,c) \rightarrow (\accentset{\circ}{\Delta}_1', a', b', c')$ and $F_2:(\accentset{\circ}{\Delta}_2, c,d,e) \rightarrow (\accentset{\circ}{\Delta}_2', c', d', e')$ can be made to agree along $c$. 
\end{lemma}

\begin{proof}
Let $f$ be the arc shown in Figure \ref{one side}. Then $f$ is an arc on $S$ which intersects $c$ exactly once and is disjoint from $a$, $b$, $d$, and $e$. Then $a$, $d$, and $f$ are the sides of a triangle on $S$. This triangle is non-embedded if $a=d$, and embedded otherwise. Let $\mathcal{K}$ be the finite simplicial subcomplex of $\arc$ from either Lemma \ref{embedded} or Lemma \ref{non-embedded} corresponding to this triangle. Further, let $\mathcal{B}$ be the finite simplicial complex of $\arc$ from Lemma \ref{detectible} corresponding to the pair $(c,f)$ and 
let $\mathcal{C}_1$ and $\mathcal{C}_2$ be the finite simplicial complexes of $\arc$ from Lemma \ref{embedded} corresponding to $\Delta_1$ and $\Delta_2$, respectively. We define: 
\[\mathcal{E} = \text{Span}_{\arc}(\{a,b,c,d,e,f\}\cup \mathcal{B} \cup \mathcal{C}_1 \cup \mathcal{C}_2 \cup \mathcal{K}).\]

Let $\mathcal{Y}$ be any simplicial subcomplex of $\arc$ containing $\mathcal{E}$ and $\lambda:\mathcal{Y} \rightarrow \arcp$ a locally injective simplicial map. By Lemma \ref{embedded} there is an embedded triangle $\Delta_1'$ on $S'$ with sides $a'$, $b'$ and $c'$ and an embedded triangle $\Delta_2'$ on $S'$ with sides $c'$, $d'$, and $e'$. Further, by Lemma \ref{detectible}, we know that $i(c', f')=1$. Since $\lambda$ is simplicial, $a'$, $b'$, $d'$, and $e'$ are disjoint from $f'$. Depending on the identifications of sides, either Lemma \ref{embedded} or Lemma \ref{non-embedded} implies that $a'$, $d'$ and $f'$ border a triangle. By inspection, we see that these conditions hold simultaneously only if the orientation of $\Delta_1'$ relative to $\Delta_2'$ is the same as the orientation of $\Delta_1$ relative to $\Delta_2$.  Figure \ref{one side bad}, for example, shows the case where no sides are identified and the relative orientations are reversed. As another example, Figure \ref{two sides a bad}, shows the case where $b=e$ and the relative orientations are reversed. In both cases, there is no possible placement of $f'$ which satisfies all the conditions above. All possible identifications of the sides of $\Delta_1$ and $\Delta_2$ yield this result, hence the natural homeomorphisms $F_1: (\accentset{\circ}{\Delta}_1, a, b, c) \rightarrow (\accentset{\circ}{\Delta}'_1, a', b', c')$ and $F_2:(\accentset{\circ}{\Delta_2}, c, d, e) \rightarrow (\accentset{\circ}{\Delta}_2', c', d', e')$ can be made to agree along $c$. 
\end{proof}

\begin{figure}[h]
    \labellist
        \pinlabel{$a$} at 230 790
        \pinlabel{$b$} at 630 790
        \pinlabel{$d$} at 230 350
        \pinlabel{$e$} at 630 360
        \pinlabel{$c$} at 430 620
        \pinlabel{$\Delta_1 \cup \Delta_2$} at 430 100
        \pinlabel{$a'$} at 1350 790
        \pinlabel{$b'$} at 1750 790
        \pinlabel{$e'$} at 1350 350
        \pinlabel{$d'$} at 1750 360
        \pinlabel{$c'$} at 1550 630
        \pinlabel{$\Delta_1' \cup \Delta_2'$} at 1550 100
    \endlabellist
    \centering
    \includegraphics[scale=.11]{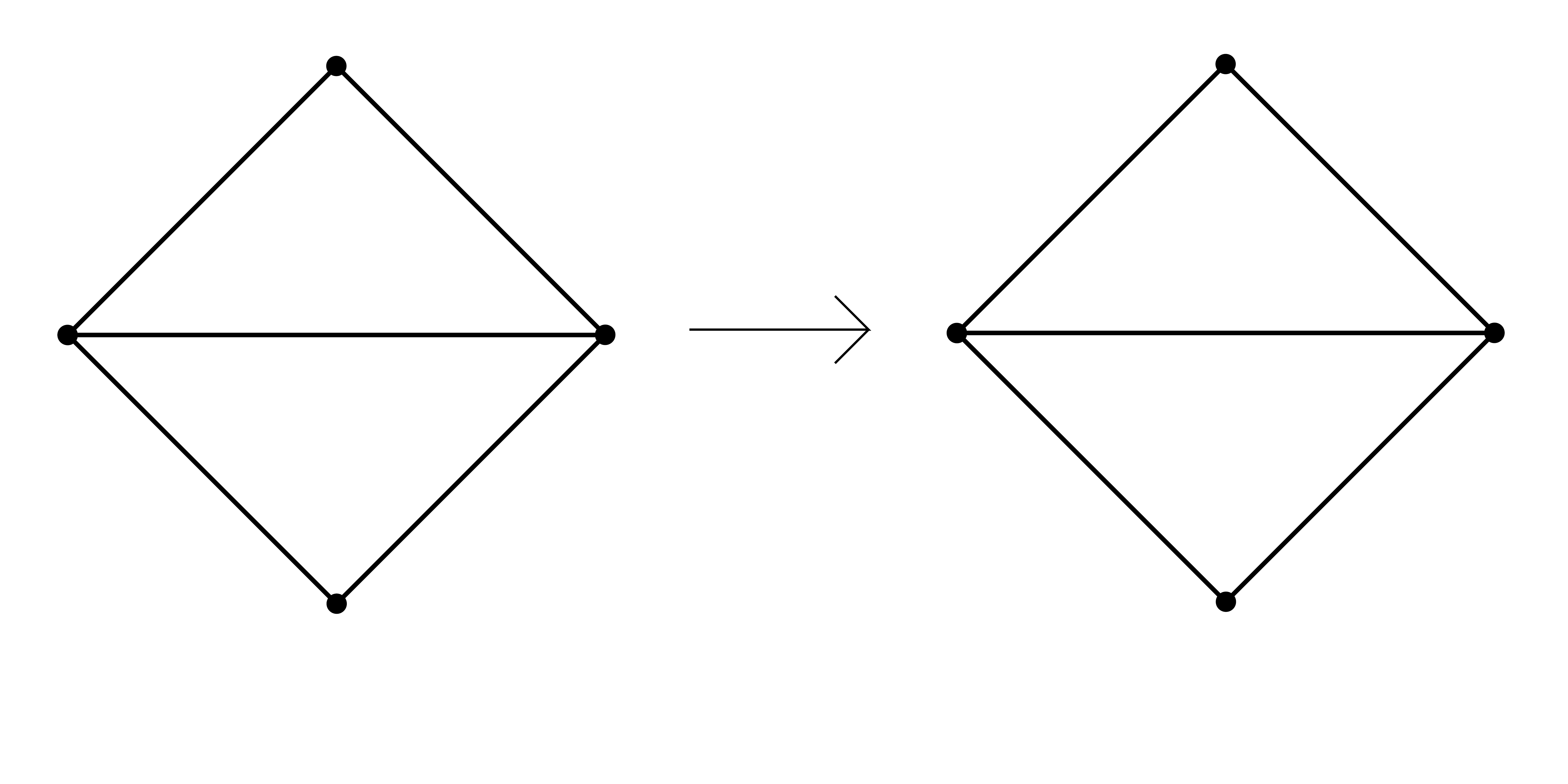}
    \caption{Triangle configurations reversing relative orientation}
    \label{one side bad}
\end{figure}

\begin{figure}[h]
    \labellist
        \pinlabel{$a$} at 170 700
        \pinlabel{$d$} at 170 280
        \pinlabel{$c$} at 250 530
        \pinlabel{$b$} at 520 540
        \pinlabel{$\Delta_1 \cup \Delta_2$} at 400 100
        \pinlabel{$a$} at 970 700
        \pinlabel{$d$} at 970 260
        \pinlabel{$c$} at 1080 585
        \pinlabel{$b$} at 1350 595
        \pinlabel{$c$} at 1080 390
        \pinlabel{$b$} at 1350 375
        \pinlabel{$\Delta_2$} at 1250 70
        \pinlabel{$\Delta_1$} at 1250 900
        \pinlabel{$a'$} at 1810 700
        \pinlabel{$d'$} at 1810 260
        \pinlabel{$c'$} at 1925 610
        \pinlabel{$b'$} at 2195 610
        \pinlabel{$b'$} at 1925 380
        \pinlabel{$c'$} at 2195 380
        \pinlabel{$\Delta_2'$} at 2095 70
        \pinlabel{$\Delta_1'$} at 2095 900
        \pinlabel{$a'$} at 2620 700
        \pinlabel{$d'$} at 2830 520
        \pinlabel{$c'$} at 2770 670
        \pinlabel{$b'$} at 2970 680
        \pinlabel{$\Delta_1' \cup \Delta_2'$} at 2870 110
    \endlabellist
    \centering
    \includegraphics[scale=.11]{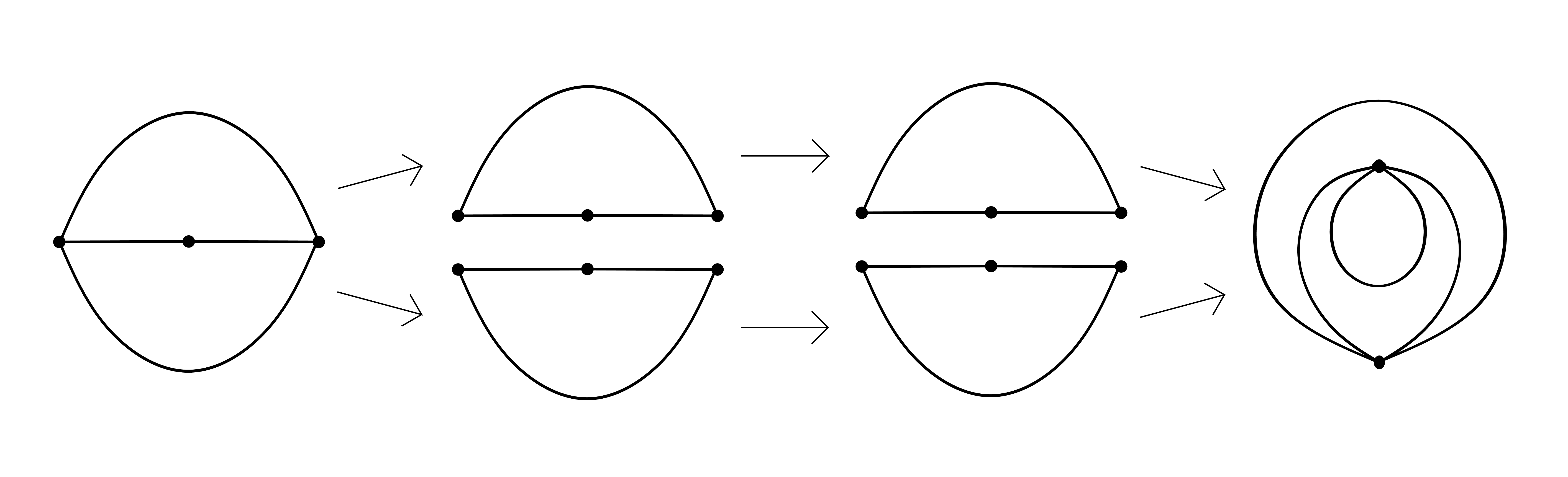}
    \caption{Triangle configurations reversing relative orientation}
    \label{two sides a bad}
\end{figure}

Now we apply the above results to find a candidate homeomorphism (cf \cite[Prop. 3.8]{irmak2010injective}).

\begin{prop}\label{extension}
Let $T$ be a triangulation of $S$. Then there exists a finite simplicial subcomplex $\mathcal{F}$ of $\arc$ containing $T$ with the following property: If $\mathcal{Y}$ is any simplicial subcomplex of $\arc$ containing $\mathcal{F}$ and $\lambda:\mathcal{Y} \rightarrow \arcp$ a locally injective simplicial map, there exists a homeomorphism $H:S\rightarrow S'$ whose induced map on $\arc$ agrees with $\lambda$ on $T$. 
\end{prop}

\begin{proof} 
There are $N=4g+2n-4$ triangles in $T$. Denote them by $\{\Delta_i:1\leq i \leq N\}$. Since $S$ is neither $S_{0,3}$ nor $S_{1,1}$, no two components $\Delta_i$, $\Delta_j$ share the same three sides. Suppose that $M$ is the number of embedded triangles in $T$. Then after reordering, we may assume that $\{\Delta_i:1\leq i \leq M\}$ are embedded triangles and $\{\Delta_j:M+1\leq j \leq N\}$ are non-embedded triangles. Then for each $1\leq i \leq M$, let $\mathcal{C}_i$ be the finite simplicial subcomplex of $\arc$ from Lemma \ref{embedded} corresponding to $\Delta_i$, and for each $M+1 \leq i \leq N$, let $\mathcal{D}_i$ be the finite simplicial subcomplex of $\arc$ from Lemma \ref{non-embedded} corresponding to $\Delta_i$. 

Suppose there are $K$ arcs in $T$ which border two distinct triangles of $T$ (the inner arc of a non-embedded triangle does not have this property, so it may be that $K<N$.) Let $\mathcal{E}_i$, $1\leq i \leq K$ be the finite simplicial complexes of $\arc$ from Lemma \ref{one side prop} corresponding to each pair. We define:
\[\mathcal{F} = \text{Span}_{\arc}\left(T \cup \bigcup_{1\leq i \leq M} \mathcal{C}_i \cup \bigcup_{M+1 \leq i \leq N} \mathcal{D}_i \cup \bigcup_{1\leq i \leq K} \mathcal{E}_i \right).\]

Let $\mathcal{Y}$ be any simplicial subcomplex of $\arc$ containing $\mathcal{F}$ and let $\lambda:\mathcal{Y} \rightarrow \arcp$ be a locally injective simplicial map. Let $T'=\lambda(T)$; note that it is a triangulation of $S'$. Since $\text{dim}(\arc)=\text{dim}(\arcp)$, there are also $N$ triangles in $T'$. Denote them by $\{\Delta'_i:1\leq i \leq N\}$. Suppose $\Delta_i$ is a triangle in $S$ with sides $a$, $b$, and $c$. Then Lemmas \ref{embedded} and \ref{non-embedded} ensure that $\lambda(a)$, $\lambda(b)$, and $\lambda(c)$ border a triangle in $T'$ of the same type (embedded or non-embedded). Then after reordering, we may assume that $\Delta'_i$ corresponds to $\Delta_i$ in this way.

Suppose $\Delta_i$ is embedded with sides $a_i$, $b_i$, and $c_i$. Let $a'_i=\lambda(a_i)$, $b'_i=\lambda(b_i)$, and $c'_i=\lambda(c_i)$. Then there exists a homeomorphism $F_i:(\accentset{\circ}{\Delta}_i, a_i, b_i, c_i)\rightarrow (\accentset{\circ}{\Delta}'_i, a'_i, b'_i, c'_i)$. This homeomorphism is well-defined up to relative isotopies and its orientation type is fixed. Suppose $\Delta_i$ is non-embedded with inner arc $a_i$ and outer arc $b_i$. Then $a'_i=\lambda(a_i)$ is the inner arc of a triangle $\Delta'_i$ with outer arc $b'_i=\lambda(b_i)$. Further, there exists two homeomorphisms $F_i,F_{i}^*:(\accentset{\circ}{\Delta}_i, a_i, b_i)\rightarrow (\accentset{\circ}{\Delta}'_i, a'_i, b'_i)$ with opposite orientation types. 

Now suppose $\Delta_i$ and $\Delta_j$ are two distinct triangles in a triangulation $T$ of $S$ which share the side $s$. Since $S$ is not $S_{0,3}$, it cannot be the case that $\Delta_i$ and $\Delta_j$ are both non-embedded triangles. Suppose $\Delta_i$ is non-embedded and $\Delta_j$ is embedded. The shared side $s$ must be the outer arc of $\Delta_i$. Then one of $F_i$ or $F_i^*$ can be made to agree with $F_j$ along $s$ and we may assume it is the one called $F_i$. If instead $\Delta_i$ and $\Delta_j$ are both embedded, then Lemma \ref{one side prop} guarantees that $F_i$ and $F_j$ can be made to agree along $s$. 

Then there is a homeomorphism of the punctured surfaces $F:S\backslash \mathscr{P}_S\rightarrow S' \backslash \mathscr{P}_{S'}$ whose restriction to $\accentset{\circ}{\Delta}_i$ is equal to $F_i$ for $1\leq i \leq N$. This can be extended uniquely to a homeomorphism $H:S\rightarrow S'$, and by construction the induced map by $H$ on $\arc$ agrees with $\lambda$ on $T$.
\end{proof}

We can now prove the general case of Theorem \ref{main theorem one general} (cf \cite[Prop. 3.11]{irmak2010injective}).

\begin{proof}[Proof of Theorem \ref{main theorem one general}.]
We dispensed with the cases where $\text{dim}(\arc)\leq 2$ in Section \ref{section: exceptional cases}. Suppose $\text{dim}(\arc)>2$.

Let $T$ be a triangulation of $S$ and let $\mathcal{F}$ be as in Lemma \ref{extension}. Suppose $y_1, \ldots, y_k$ are the vertices of $\mathcal{F}$. Then for each $i$, we fix a triangulation $T^i$ of $S$ containing $y_i$. By Proposition \ref{prop: connectedness}, there exists a finite sequence of triangulations $T=T^i_0, T^i_1, \ldots, T^i_{m_i-1}, T^i_{m_i} = T^i$ such that for each $0\leq j \leq m_i-1$, $T^i_{j}$ and $T^i_{j+1}$ differ by a flip. We define:
\[\mathcal{X} = \text{Span}_{\arc}\left(\mathcal{F} \cup \bigcup_{1\leq i \leq k} \bigcup_{0\leq j \leq m_i} \{T^i_j\}\right).\]
Then $\mathcal{X}$ is a finite simplicial subcomplex of $\arc$ and has the property that any vertex $y\in \mathcal{X}$ is contained in a triangulation $T^*$ of vertices in $\mathcal{X}$ such that there exists a finite sequence of simplices $T=T_0, T_1, \ldots, T_{m-1}, T_m = T^*$, all contained in $\mathcal{X}$, where for each $0\leq j \leq m-1$, $T_{j}$ and $T_{j+1}$ differ by a flip.

Let $\lambda:\mathcal{X} \rightarrow \arcp$ be any locally injective simplicial map and $H:S\rightarrow S'$ be as in Proposition \ref{extension}. Let $H_*:\arc \rightarrow \arcp$ be the induced map of $H$. Proposition \ref{extension} says that $H_*(x)=\lambda(x)$ for $x\in T$. We now show that this is true for any vertex $y\in \mathcal{X}$ so that $H_*|_\mathcal{X}=\lambda$. 

Suppose $\phi = (H_*)^{-1} \circ \lambda: \mathcal{X} \rightarrow \arc$. We already know that $\phi$ is the identity on $T$. Suppose $y$ is a vertex in $\mathcal{X}$. Let $T=T_0, T_1, \ldots, T_{m-1}, T_m = T^*$ be the finite sequence of triangulations, all contained in $\mathcal{X}$, such that $y\in T^*$ and for each $0\leq j \leq m-1$, $T_{j}$ and $T_{j+1}$ differ by a flip. We will show that if $\phi$ is the identity on $T_i$, this implies that it is the identity on $T_{i+1}$. Now $T_i \cap T_{i+1}$ has codimension one in $\arc$. Say $a$ is the single arc in $T_i\backslash T_{i+1}$ and $b$ the single arc in $T_{i+1}\backslash T_i$. In general, a collection of arcs corresponding to a codimension one simplex in $\arc$ has either one or two arcs disjoint from all arcs in the collection, ie one or two ways to complete the collection into a triangulation of $S$. For $T_i\cap T_{i+1}$, there are two arcs --- $a$ and $b$. We know that $\phi (a)=a$ and $\phi(T_i\cap T_{i+1})=T_i\cap T_{i+1}$ by our hypothesis. Then since $\phi$ is a locally injective simplicial map, it must be the case that $\phi(b)=b$, hence $\phi$ is the identity on $T_{i+1}$. Then by induction, it follows that $\phi(y)=y$ and hence that $\phi$ is the identity on all of $\mathcal{X}$. This implies that $H_*|_\mathcal{X}=\lambda$. If $H'$ is another homeomorphism whose induced map agrees with $\lambda$ on $\mathcal{X}$, then $H$ and $H'$ agree on the triangulation $T$, thus they are homotopic. 
\end{proof}

We can now prove the following corollaries. 

\begin{proof}[Proof of Corollary \ref{generalization of im}.] 
Fix $S$, $S'$ and $\phi$. Let $\mathcal{X}$ be the finite rigid set from Theorem \ref{main theorem one general}. Then $\phi|_{\mathcal{X}}:\mathcal{X}\rightarrow \arcp$ meets the conditions described in Theorem 1.1, so there is a homeomorphism $H:S\rightarrow S'$ which induces $\phi|_{\mathcal{X}}$. Then it remains to show that $\phi$ and $H$ agree on all of $\arc$, which can be done by employing a near identical argument to the one used in the proof of Theorem \ref{main theorem one general}.

\end{proof}

\begin{proof}[Proof of Corollary \ref{corollary to main theorem one}.]
Suppose $\varphi:\arc \rightarrow \arcp$ is an isomorphism. If $S\neq S_{0,3}$, we apply Theorem \ref{main theorem one general} to the injective simplicial map $\varphi|_\mathcal{X}:\mathcal{X}\rightarrow \arcp$, where $\mathcal{X}$ is the finite rigid set given in the theorem. Suppose $S=S_{0,3}$. Then $\text{dim}(\arc)=2$, so $S'$ is either $S_{0,3}$ or $S_{1,1}$. However $\mathcal{A}(S_{0,3})$ is finite and $\mathcal{A}(S_{1,1})$ is infinite, so it must be that $S'$ is $S_{0,3}$, hence $S$ and $S'$ are homeomorphic. 
\end{proof}

Finally, we extend the proof of Theorem \ref{main theorem one general} given above to construct an exhaustion of $\arc$ by finite rigid sets.

\begin{proof}[Proof of Theorem \ref{main theorem two}.]
We dispensed with the cases where $\text{dim}(\arc)\leq 2$ in Section \ref{section: exceptional cases}, so suppose $\text{dim}(\arc)> 2$. Let $\mathcal{X}_0$ be the finite rigid set of $\arc$ from Theorem \ref{main theorem one general}, which by construction contains a triangulation $T$ of $S$. It is well-known that $\arc$ has countable vertex set, so we can enumerate the vertices $\mathcal{A}^{(0)}(S)=\{x_1, x_2, \ldots\}$. As above, for each $i\in \N$, there is a triangulation $T^{i}$ of $S$ containing $x_i$ and a finite sequence of simplices $T=T^{i}_0, T^{i}_1, \ldots, T^{i}_{m_i-1}, T^{i}_{m_i} = T^{i}$ where for each $0\leq j \leq m_i-1$, $T^{i}_{j}$ and $T^{i}_{j+1}$ differ by a flip. Then for $i\geq 1$, we define:
\[\mathcal{X}_i = \text{Span}_{\arc}\left(\mathcal{X}_{i-1} \cup \bigcup_{0\leq j \leq m_i} T^{i}_j \right).\]

Observe that $\mathcal{X}_i$ is finite. An identical argument to the one employed in the proof of Theorem \ref{main theorem one general} above shows that $\mathcal{X}_i$ is rigid. It is clear that $\mathcal{X}_0\subseteq \mathcal{X}_1 \subseteq ... \subseteq \arc$ and $\bigcup_{i\in \N} \mathcal{X}_i = \arc$. 

\end{proof}

\printbibliography

\end{document}